\documentclass{amsart}
\usepackage{amssymb}
\usepackage{amsmath}
\usepackage{amsfonts}

\newtheorem{theorem}{Theorem}
\theoremstyle{plain}

\newtheorem{corollary}{Corollary}

\newtheorem{definition}{Definition}

\newtheorem{lemma}{Lemma}

\newtheorem{remark}{Remark}

\numberwithin{equation}{section}
\input{tcilatex}

\begin{document}
\title[FRACTIONAL INEQUALITIES]{FRACTIONAL INTEGRAL INEQUALITIES VIA $s-$%
CONVEX FUNCTIONS}
\author{M.EM\.{I}N \"{O}ZDEM\.{I}R$^{\blacktriangle }$}
\address{$^{\blacktriangle }$ATAT\"{U}RK UNIVERSITY, K.K. EDUCATION FACULTY,
DEPARTMENT OF MATHEMATICS, 25240, CAMPUS, ERZURUM, TURKEY}
\email{emos@atauni.edu.tr}
\author{HAVVA KAVURMACI$^{\blacktriangle }$}
\email{hkavurmaci@atauni.edu.tr}
\author{\c{C}ET\.{I}N YILDIZ$^{\blacktriangle ,\bigstar }$}
\email{yildizc@atauni.edu.tr}
\thanks{$^{\bigstar }$Corresponding Author.}
\date{January 23, 2012}

\begin{abstract}
In this paper, we establish several inequalities for $s-$convex mappings
that are connected with the Riemann-Liouville fractional integrals. Our
results have some relationships with certain integral inequalities in the
literature.
\end{abstract}

\maketitle

\section{INTRODUCTION}

Let $f:I\subseteq 
\mathbb{R}
\rightarrow 
\mathbb{R}
$ be a convex function defined on the interval $I$ of real numbers and $a<b.$
The following double inequality%
\begin{equation*}
f\left( \frac{a+b}{2}\right) \leq \frac{1}{b-a}\dint\limits_{a}^{b}f(x)dx%
\leq \frac{f(a)+f(b)}{2}
\end{equation*}%
is well known in the literature as Hadamard's inequality. Both inequalities
hold in the reversed direction if $f$ is concave.

Let real function $f$ be defined on some nonempty interval $I$ of real line $%
\mathbb{R}
.$ The function $f$ is said to be convex on $I$ if inequality%
\begin{equation*}
f(tx+(1-t)y)\leq tf(x)+(1-t)f(y)
\end{equation*}%
holds for all $x,y\in I$ and $t\in \lbrack 0,1].$

In \cite{OR}, $s-$convex functions defined by Orlicz as following.

\begin{definition}
A function $f:%
\mathbb{R}
^{+}\rightarrow 
\mathbb{R}
,$ where $%
\mathbb{R}
^{+}=[0,\infty ),$ is said to be $s-$convex in the first sense if 
\begin{equation*}
f(\alpha x+\beta y)\leq \alpha ^{s}f(x)+\beta ^{s}f(y)
\end{equation*}%
for all $x,y\in \lbrack 0,\infty ),$ $\alpha ,\beta \geq 0$ with $\alpha
^{s}+\beta ^{s}=1$ and for some fixed $s\in (0,1].$ We denote by $K_{s}^{1}$
the class of all $s-$convex functions$.$
\end{definition}

\begin{definition}
A function $f:%
\mathbb{R}
^{+}\rightarrow 
\mathbb{R}
,$ where $%
\mathbb{R}
^{+}=[0,\infty ),$ is said to be $s-$convex in the second sense if 
\begin{equation*}
f(\alpha x+\beta y)\leq \alpha ^{s}f(x)+\beta ^{s}f(y)
\end{equation*}%
for all $x,y\in \lbrack 0,\infty ),$ $\alpha ,\beta \geq 0$ with $\alpha
+\beta =1$ and for some fixed $s\in (0,1].$ We denote by $K_{s}^{2}$ the
class of all $s-$convex functions.
\end{definition}

Orlicz defined these class of functions in \cite{OR} and these definitions
was used in the theory of Orlicz spaces in \cite{MO} and \cite{JO}.
Obviously, one can see that if we choose $s=1$, both definitions reduced to
ordinary concept of convexity.

For several results related to above definitions we refer readers to \cite%
{HM}, \cite{SF}, \cite{UK} and \cite{SEL}.

In \cite{SF}, Hadamard's inequality which for $s-$convex functions in the
second sence is proved by Dragomir and Fitzpatrick.

\begin{theorem}
Suppose that $f:[0,\infty )\rightarrow \lbrack 0,\infty )$ is an $s-$convex
function in the second sence, where $s\in (0,1)$ and let $a,b\in \lbrack
0,\infty ),$ $a<b.$ If $f\in L_{^{_{1}}}([a,b]),$ then the following
inequalities hold:%
\begin{equation}
2^{s-1}f\left( \frac{a+b}{2}\right) \leq \frac{1}{b-a}\underset{a}{\overset{b%
}{\int }}f(x)dx\leq \frac{f(a)+f(b)}{s+1}.  \label{1}
\end{equation}%
The constant $k=\frac{1}{s+1}$ is the best possible in the second inequality
in (\ref{1}).
\end{theorem}

In \cite{UK}, K\i rmac\i\ et al. obtained Hadamard type inequalities which
holds for $s-$convex functions in the second sence. It is given in the next
theorem.

\begin{theorem}
\label{havv} Let $f:I\rightarrow 
\mathbb{R}
,$ $I\subset \lbrack 0,\infty )$, be differentiable function on $I^{\circ }$
such that $f^{\prime }\in L_{1}([a,b]),$ where $a,b\in I$, $a<b.$ If $%
\left\vert f^{\prime }\right\vert ^{q}$ is $s-$convex on $[a,b]$ for some
fixed $s\in (0,1)$ and $q\geq 1,$ then:%
\begin{equation}
\left\vert \frac{f(a)+f(b)}{2}-\frac{1}{b-a}\underset{a}{\overset{b}{\int }}%
f(x)dx\right\vert \leq \frac{b-a}{2}\left( \frac{1}{2}\right) ^{\frac{q-1}{q}%
}\left[ \frac{s+\left( \frac{1}{2}\right) ^{s}}{(s+1)(s+2)}\right] ^{\frac{1%
}{q}}\left[ \left\vert f^{\prime }(a)\right\vert ^{q}+\left\vert f^{\prime
}(b)\right\vert ^{q}\right] ^{\frac{1}{q}}.  \label{2}
\end{equation}
\end{theorem}

In \cite{DA}, Dragomir and Agarwal proved the following inequality.

\begin{theorem}
Let $f:I%
{{}^\circ}%
\subseteq 
\mathbb{R}
\rightarrow 
\mathbb{R}
$ be a differentiable mapping on $I^{\circ }$,$a,b\in I^{\circ }$ with $a<b$%
, and let $p>1$. If the new mapping $\left\vert f^{\prime }\right\vert ^{%
\frac{p}{p-1}}$ is convex on $[a,b],$ then the \pounds ollowing inequality
holds:%
\begin{equation}
\left\vert \frac{f(a)+f(b)}{2}-\frac{1}{b-a}\underset{a}{\overset{b}{\int }}%
f(x)dx\right\vert \leq \frac{b-a}{2\left( p+1\right) ^{\frac{1}{p}}}\left[ 
\frac{\left\vert f^{\prime }(a)\right\vert ^{\frac{p}{p-1}}+\left\vert
f^{\prime }(b)\right\vert ^{\frac{p}{p-1}}}{2}\right] ^{\frac{p-1}{p}}.
\label{4}
\end{equation}
\end{theorem}

In \cite{sarikayaz}, Set et al. proved the following Hadamard type
inequality for $s-$convex functions in the second sense via
Riemann-Liouville fractional integral.

\begin{theorem}
\label{havva} Let $f:\left[ a,b\right] \subset \lbrack 0,\infty )\rightarrow 
\mathbb{R}
$ be a differentiable mapping on $\left( a,b\right) $ with $a<b$ such that $%
f^{\prime }\in L[a,b].$ If $\left\vert f^{\prime }\right\vert ^{q}$ is $s-$%
convex in the second sense on $[a,b]$ for some fixed $s\in (0,1]$ and $q\geq
1,$ then the following inequality for fractional integrals holds%
\begin{eqnarray}
&&  \label{3} \\
&&\left\vert \frac{f(a)+f(b)}{2}-\frac{\Gamma (\alpha +1)}{2(b-a)^{\alpha }}%
\left[ J_{a^{+}}^{\alpha }f(b)+J_{b^{-}}^{\alpha }f(a)\right] \right\vert 
\notag \\
&\leq &\frac{b-a}{2}\left[ \frac{2}{\alpha +1}\left( 1-\frac{1}{2^{\alpha }}%
\right) \right] ^{1-\frac{1}{q}}  \notag \\
&&\times \left[ \beta \left( \frac{1}{2},s+1,\alpha +1\right) -\beta \left( 
\frac{1}{2},\alpha +1,s+1\right) +\frac{2^{\alpha +s}-1}{\left( \alpha
+s+1\right) 2^{\alpha +s}}\right] ^{\frac{1}{q}}\left( \left\vert f^{\prime
}(a)\right\vert ^{q}+\left\vert f^{\prime }(b)\right\vert ^{q}\right) ^{%
\frac{1}{q}}.  \notag
\end{eqnarray}
\end{theorem}

Now, we give some necessary definitions and mathematical preliminaries of
fractional calculus theory which are used throughout this paper, see(\cite%
{samko}).

\begin{definition}
Let $f\in L_{1}[a,b].$ The Riemann-Liouville integrals $J_{a^{+}}^{\alpha }f$
and $J_{b^{-}}^{\alpha }f$ of order $\alpha >0$ with $a\geq 0$ are defined by%
\begin{equation*}
J_{a^{+}}^{\alpha }f(x)=\frac{1}{\Gamma (\alpha )}\underset{a}{\overset{x}{%
\int }}\left( x-t\right) ^{\alpha -1}f(t)dt,\text{ \ }x>a
\end{equation*}%
and%
\begin{equation*}
J_{b^{-}}^{\alpha }f(x)=\frac{1}{\Gamma (\alpha )}\underset{x}{\overset{b}{%
\int }}\left( t-x\right) ^{\alpha -1}f(t)dt,\text{ \ }x<b
\end{equation*}%
respectively where $\Gamma (\alpha )=\underset{0}{\overset{\infty }{\int }}%
e^{-u}u^{\alpha -1}du.$ Here is $J_{a^{+}}^{0}f(x)=J_{b^{-}}^{0}f(x)=f(x).$
\end{definition}

In the case of $\alpha =1$, the fractional integral reduces to the classical
integral. For some recent results connected with \ fractional integral
inequalities see (\cite{zekiiss}-\cite{dahtab}).

In order to prove our main theorems, we need the following lemma:

\begin{lemma}
(see \cite{cetin}) Let $f:I\subset 
\mathbb{R}
\rightarrow 
\mathbb{R}
$ be a differentiable mapping on $I$ with $a<r,$ $a,r\in I.$ If $\ f^{\prime
}\in L[a,r],$ then the following equality for fractional integral holds:%
\begin{eqnarray*}
&&\frac{f(a)+f(r)}{2}-\frac{\Gamma (\alpha +1)}{2(r-a)^{\alpha }}\left[
J_{a^{+}}^{\alpha }f(r)+J_{r^{-}}^{\alpha }f(a)\right] \\
&=&\frac{r-a}{2}\int_{0}^{1}\left[ (1-t)^{\alpha }-t^{\alpha }\right]
f^{\prime }(r+(a-r)t)dt.
\end{eqnarray*}
\end{lemma}

\section{MAIN RESULTS}

\begin{theorem}
\label{cet} Let $f:\left[ a,b\right] \subset \lbrack 0,\infty )\rightarrow 
\mathbb{R}
$ be a differentiable mapping on $\left( a,b\right) $ with $a<r\leq b$ such
that $\ f^{\prime }\in L[a,b].$ If $\left\vert f^{\prime }\right\vert $ is $%
s-$convex on $[a,b]$ for some fixed $s\in (0,1]$, then the following
inequality for fractional integrals holds%
\begin{eqnarray*}
&&\left\vert \frac{f(a)+f(r)}{2}-\frac{\Gamma (\alpha +1)}{2(r-a)^{\alpha }}%
\left[ J_{a^{+}}^{\alpha }f(r)+J_{r^{-}}^{\alpha }f(a)\right] \right\vert \\
&\leq &\frac{r-a}{2}\left[ \beta \left( \frac{1}{2},s+1,\alpha +1\right)
-\beta \left( \frac{1}{2},\alpha +1,s+1\right) +\frac{2^{\alpha +s}-1}{%
\left( \alpha +s+1\right) 2^{\alpha +s}}\right] \left[ \left\vert f^{\prime
}(a)\right\vert +\left\vert f^{\prime }(r)\right\vert \right] .
\end{eqnarray*}
\end{theorem}

\begin{proof}
From Lemma 1 and using the properties of modulus, we get%
\begin{eqnarray*}
&&\left\vert \frac{f(a)+f(r)}{2}-\frac{\Gamma (\alpha +1)}{2(r-a)^{\alpha }}%
\left[ J_{a^{+}}^{\alpha }f(r)+J_{r^{-}}^{\alpha }f(a)\right] \right\vert \\
&\leq &\frac{r-a}{2}\int_{0}^{1}\left\vert (1-t)^{\alpha }-t^{\alpha
}\right\vert \left\vert f^{\prime }(r+(a-r)t)\right\vert dt.
\end{eqnarray*}%
Since $\left\vert f^{\prime }\right\vert $ is $s-$convex on $\left[ a,b%
\right] $, we obtain inequality%
\begin{equation*}
\left\vert f^{\prime }(r+(a-r)t)\right\vert =\left\vert f^{\prime
}(ta+(1-t)r)\right\vert \leq t^{s}\left\vert f^{\prime }(a)\right\vert
+(1-t)^{s}\left\vert f^{\prime }(r)\right\vert ,\text{ }t\in (0,1).
\end{equation*}%
Hence,%
\begin{eqnarray*}
&&\left\vert \frac{f(a)+f(r)}{2}-\frac{\Gamma (\alpha +1)}{2(r-a)^{\alpha }}%
\left[ J_{a^{+}}^{\alpha }f(r)+J_{r^{-}}^{\alpha }f(a)\right] \right\vert \\
&\leq &\frac{r-a}{2}\left\{ \int_{0}^{\frac{1}{2}}\left[ (1-t)^{\alpha
}-t^{\alpha }\right] \left[ t^{s}\left\vert f^{\prime }(a)\right\vert
+(1-t)^{s}\left\vert f^{\prime }(r)\right\vert \right] dt\right. \\
&&\text{ \ \ \ \ }+\left. \int_{\frac{1}{2}}^{1}\left[ t^{\alpha
}-(1-t)^{\alpha }\right] \left[ t^{s}\left\vert f^{\prime }(a)\right\vert
+(1-t)^{s}\left\vert f^{\prime }(r)\right\vert \right] dt\right\}
\end{eqnarray*}%
and%
\begin{eqnarray*}
\int_{0}^{\frac{1}{2}}t^{s}(1-t)^{\alpha }dt &=&\int_{\frac{1}{2}%
}^{1}(1-t)^{s}t^{\alpha }dt=\beta \left( \frac{1}{2};s+1,\alpha +1\right) ,
\\
\int_{0}^{\frac{1}{2}}(1-t)^{s}t^{\alpha }dt &=&\int_{\frac{1}{2}%
}^{1}t^{s}(1-t)^{\alpha }dt=\beta \left( \frac{1}{2};\alpha +1,s+1\right) ,
\\
\int_{0}^{\frac{1}{2}}t^{s+\alpha }dt &=&\int_{\frac{1}{2}%
}^{1}(1-t)^{s+\alpha }dt=\frac{1}{2^{s+\alpha +1}(s+\alpha +1)},
\end{eqnarray*}%
\begin{equation*}
\int_{0}^{\frac{1}{2}}(1-t)^{s+\alpha }dt=\int_{\frac{1}{2}}^{1}t^{s+\alpha
}dt=\frac{1}{s+\alpha +1}-\frac{1}{2^{s+\alpha +1}(s+\alpha +1)}.
\end{equation*}%
We obtain%
\begin{eqnarray*}
&&\left\vert \frac{f(a)+f(r)}{2}-\frac{\Gamma (\alpha +1)}{2(r-a)^{\alpha }}%
\left[ J_{a^{+}}^{\alpha }f(r)+J_{r^{-}}^{\alpha }f(a)\right] \right\vert \\
&\leq &\frac{r-a}{2}\left[ \beta \left( \frac{1}{2},s+1,\alpha +1\right)
-\beta \left( \frac{1}{2},\alpha +1,s+1\right) +\frac{2^{\alpha +s}-1}{%
\left( \alpha +s+1\right) 2^{\alpha +s}}\right] \left[ \left\vert f^{\prime
}(a)\right\vert +\left\vert f^{\prime }(r)\right\vert \right] .
\end{eqnarray*}
\end{proof}

\begin{theorem}
\label{cett} Let $f:[a,b]\subset \lbrack 0,\infty )\rightarrow 
\mathbb{R}
$ be a differentiable mapping on $(a,b)$ with $a<r\leq b$ such that $%
f^{\prime }\in L[a,b].$ If $\left\vert f^{\prime }\right\vert ^{q}$ is $s-$%
convex in the second sense on $[a,b]$ for some fixed $s\in (0,1]$ and $q>1$
with $\frac{1}{p}+\frac{1}{q}=1,$ then the following inequality for
fractional integrals holds%
\begin{eqnarray*}
&&\left\vert \frac{f(a)+f(r)}{2}-\frac{\Gamma (\alpha +1)}{2(r-a)^{\alpha }}%
\left[ J_{a^{+}}^{\alpha }f(r)+J_{r^{-}}^{\alpha }f(a)\right] \right\vert \\
&\leq &\frac{r-a}{2}\left( \frac{1}{\alpha p+1}\right) ^{\frac{1}{p}}\left( 
\frac{\left\vert f^{\prime }(a)\right\vert ^{q}+\left\vert f^{\prime
}(r)\right\vert ^{q}}{s+1}\right) ^{\frac{1}{q}}
\end{eqnarray*}%
where $\alpha \in \lbrack 0,1].$
\end{theorem}

\begin{proof}
By Lemma 1 and using H\"{o}lder inequality with the properties of modulus,
we have 
\begin{eqnarray*}
&&\left\vert \frac{f(a)+f(r)}{2}-\frac{\Gamma (\alpha +1)}{2(r-a)^{\alpha }}%
\left[ J_{a^{+}}^{\alpha }f(r)+J_{r^{-}}^{\alpha }f(a)\right] \right\vert \\
&\leq &\frac{r-a}{2}\int_{0}^{1}\left\vert (1-t)^{\alpha }-t^{\alpha
}\right\vert \left\vert f^{\prime }(r+(a-r)t)\right\vert dt \\
&\leq &\frac{r-a}{2}\left( \underset{0}{\overset{1}{\int }}\left\vert \left(
1-t\right) ^{\alpha }-t^{\alpha }\right\vert ^{p}dt\right) ^{\frac{1}{p}%
}\left( \underset{0}{\overset{1}{\int }}\left\vert f^{\prime
}(r+(a-r)t)\right\vert ^{q}dt\right) ^{\frac{1}{q}}.
\end{eqnarray*}%
We know that for $\alpha \in \lbrack 0,1]$ and $\forall t_{1},t_{2}\in
\lbrack 0,1],$%
\begin{equation*}
\left\vert t_{1}^{\alpha }-t_{2}^{\alpha }\right\vert \leq \left\vert
t_{1}-t_{2}\right\vert ^{\alpha },
\end{equation*}%
therefore%
\begin{eqnarray*}
\underset{0}{\overset{1}{\int }}\left\vert \left( 1-t\right) ^{\alpha
}-t^{\alpha }\right\vert ^{p}dt &\leq &\underset{0}{\overset{1}{\int }}%
\left\vert 1-2t\right\vert ^{\alpha p}dt \\
&=&\underset{0}{\overset{\frac{1}{2}}{\int }}\left[ 1-2t\right] ^{\alpha
p}dt+\underset{\frac{1}{2}}{\overset{1}{\int }}\left[ 2t-1\right] ^{\alpha
p}dt \\
&=&\frac{1}{\alpha p+1}.
\end{eqnarray*}%
Since $\left\vert f^{\prime }\right\vert ^{q}$ is $s-$convex on $[a,b],$we
get%
\begin{eqnarray*}
&&\left\vert \frac{f(a)+f(r)}{2}-\frac{\Gamma (\alpha +1)}{2(r-a)^{\alpha }}%
\left[ J_{a^{+}}^{\alpha }f(r)+J_{r^{-}}^{\alpha }f(a)\right] \right\vert \\
&\leq &\frac{r-a}{2}\left( \frac{1}{\alpha p+1}\right) ^{\frac{1}{p}}\left( 
\underset{0}{\overset{1}{\int }}\left[ t^{s}\left\vert f^{\prime
}(a)\right\vert ^{q}+(1-t)^{s}\left\vert f^{\prime }(r)\right\vert ^{q}%
\right] dt\right) ^{\frac{1}{q}} \\
&=&\frac{r-a}{2}\left( \frac{1}{\alpha p+1}\right) ^{\frac{1}{p}}\left( 
\frac{\left\vert f^{\prime }(a)\right\vert ^{q}+\left\vert f^{\prime
}(r)\right\vert ^{q}}{s+1}\right) ^{\frac{1}{q}}
\end{eqnarray*}%
which completes the proof.
\end{proof}

\begin{corollary}
If in\ Theorem \ref{cett}, we choose $r=b$ then, we have%
\begin{eqnarray}
&&\left\vert \frac{f(a)+f(b)}{2}-\frac{\Gamma (\alpha +1)}{2(b-a)^{\alpha }}%
\left[ J_{a^{+}}^{\alpha }f(b)+J_{b^{-}}^{\alpha }f(a)\right] \right\vert
\label{5} \\
&\leq &\frac{b-a}{2}\left( \frac{1}{\alpha p+1}\right) ^{\frac{1}{p}}\left( 
\frac{\left\vert f^{\prime }(a)\right\vert ^{q}+\left\vert f^{\prime
}(b)\right\vert ^{q}}{s+1}\right) ^{\frac{1}{q}}.  \notag
\end{eqnarray}
\end{corollary}

\begin{remark}
If we choose $\alpha =1$ ve $s=1$ in\ Corollary \ref{cett} then, we have 
\begin{equation*}
\left\vert \frac{f(a)+f(b)}{2}-\frac{1}{b-a}\underset{a}{\overset{b}{\int }}%
f(x)dx\right\vert \leq \frac{b-a}{2\left( p+1\right) ^{\frac{1}{p}}}\left[ 
\frac{\left\vert f^{\prime }(a)\right\vert ^{q}+\left\vert f^{\prime
}(b)\right\vert ^{q}}{2}\right] ^{\frac{1}{q}}
\end{equation*}%
which is the inequality in (\ref{4}).
\end{remark}

\begin{theorem}
\label{6} Let $f:\left[ a,b\right] \subset \lbrack 0,\infty )\rightarrow 
\mathbb{R}
$ be a differentiable mapping on $\left( a,b\right) $ with $a<r\leq b$ such
that $f^{\prime }\in L[a,b].$ If $\left\vert f^{\prime }\right\vert ^{q}$ is 
$s-$convex in the second sense on $[a,b]$ for some fixed $s\in (0,1]$ and $%
q\geq 1,$ then the following inequality for fractional integrals holds%
\begin{eqnarray*}
&&\left\vert \frac{f(a)+f(r)}{2}-\frac{\Gamma (\alpha +1)}{2(r-a)^{\alpha }}%
\left[ J_{a^{+}}^{\alpha }f(r)+J_{r^{-}}^{\alpha }f(a)\right] \right\vert \\
&\leq &\frac{r-a}{2}\left[ \frac{2}{\alpha +1}\left( 1-\frac{1}{2^{\alpha }}%
\right) \right] ^{1-\frac{1}{q}} \\
&&\times \left[ \beta \left( \frac{1}{2},s+1,\alpha +1\right) -\beta \left( 
\frac{1}{2},\alpha +1,s+1\right) +\frac{2^{\alpha +s}-1}{\left( \alpha
+s+1\right) 2^{\alpha +s}}\right] ^{\frac{1}{q}}\left( \left\vert f^{\prime
}(a)\right\vert ^{q}+\left\vert f^{\prime }(r)\right\vert ^{q}\right) ^{%
\frac{1}{q}}.
\end{eqnarray*}
\end{theorem}

\begin{proof}
From Lemma 1 and using the well-known power mean inequality with the
properties of modulus, we have%
\begin{eqnarray*}
&&\left\vert \frac{f(a)+f(r)}{2}-\frac{\Gamma (\alpha +1)}{2(r-a)^{\alpha }}%
\left[ J_{a^{+}}^{\alpha }f(r)+J_{r^{-}}^{\alpha }f(a)\right] \right\vert \\
&\leq &\frac{r-a}{2}\int_{0}^{1}\left\vert (1-t)^{\alpha }-t^{\alpha
}\right\vert \left\vert f^{\prime }(r+(a-r)t)\right\vert dt \\
&\leq &\frac{r-a}{2}\left( \underset{0}{\overset{1}{\int }}\left\vert \left(
1-t\right) ^{\alpha }-t^{\alpha }\right\vert dt\right) ^{1-\frac{1}{q}%
}\left( \underset{0}{\overset{1}{\int }}\left\vert \left( 1-t\right)
^{\alpha }-t^{\alpha }\right\vert \left\vert f^{\prime
}(r+(a-r)t)\right\vert ^{q}dt\right) ^{\frac{1}{q}}
\end{eqnarray*}%
On the other hand, we have%
\begin{eqnarray*}
\int_{0}^{1}\left\vert (1-t)^{\alpha }-t^{\alpha }\right\vert dt
&=&\int_{0}^{\frac{1}{2}}\left[ (1-t)^{\alpha }-t^{\alpha }\right] dt+\int_{%
\frac{1}{2}}^{1}\left[ t^{\alpha }-(1-t)^{\alpha }\right] dt \\
&=&\frac{2}{\alpha +1}\left( 1-\frac{1}{2^{\alpha }}\right) .
\end{eqnarray*}%
Since $\left\vert f^{\prime }\right\vert ^{q}$ is $s-$convex on $\left[ a,b%
\right] $, we obtain%
\begin{equation*}
\left\vert f^{\prime }(r+(a-r)t)\right\vert ^{q}=\left\vert f^{\prime
}(ta+(1-t)r)\right\vert ^{q}\leq t^{s}\left\vert f^{\prime }(a)\right\vert
^{q}+(1-t)^{s}\left\vert f^{\prime }(r)\right\vert ^{q},\text{ }t\in (0,1)
\end{equation*}%
and 
\begin{eqnarray*}
\int_{0}^{1}\left\vert (1-t)^{\alpha }-t^{\alpha }\right\vert \left\vert
f^{\prime }(r+(a-r)t)\right\vert ^{q}dt &\leq &\int_{0}^{1}\left\vert
(1-t)^{\alpha }-t^{\alpha }\right\vert \left[ t^{s}\left\vert f^{\prime
}(a)\right\vert ^{q}+(1-t)^{s}\left\vert f^{\prime }(r)\right\vert ^{q}%
\right] dt \\
&=&\int_{0}^{\frac{1}{2}}\left[ (1-t)^{\alpha }-t^{\alpha }\right] \left[
t^{s}\left\vert f^{\prime }(a)\right\vert ^{q}+(1-t)^{s}\left\vert f^{\prime
}(r)\right\vert ^{q}\right] dt \\
&&+\int_{\frac{1}{2}}^{1}\left[ t^{\alpha }-(1-t)^{\alpha }\right] \left[
t^{s}\left\vert f^{\prime }(a)\right\vert ^{q}+(1-t)^{s}\left\vert f^{\prime
}(r)\right\vert ^{q}\right] dt
\end{eqnarray*}%
Since%
\begin{eqnarray*}
\int_{0}^{\frac{1}{2}}t^{s}(1-t)^{\alpha }dt &=&\int_{\frac{1}{2}%
}^{1}(1-t)^{s}t^{\alpha }dt=\beta \left( \frac{1}{2};s+1,\alpha +1\right) ,
\\
\int_{0}^{\frac{1}{2}}(1-t)^{s}t^{\alpha }dt &=&\int_{\frac{1}{2}%
}^{1}t^{s}(1-t)^{\alpha }dt=\beta \left( \frac{1}{2};\alpha +1,s+1\right) ,
\\
\int_{0}^{\frac{1}{2}}t^{s+\alpha }dt &=&\int_{\frac{1}{2}%
}^{1}(1-t)^{s+\alpha }dt=\frac{1}{2^{s+\alpha +1}(s+\alpha +1)}
\end{eqnarray*}%
and%
\begin{equation*}
\int_{0}^{\frac{1}{2}}(1-t)^{s+\alpha }dt=\int_{\frac{1}{2}}^{1}t^{s+\alpha
}dt=\frac{1}{s+\alpha +1}-\frac{1}{2^{s+\alpha +1}(s+\alpha +1)}.
\end{equation*}%
Therefore, we have%
\begin{eqnarray*}
&&\left\vert \frac{f(a)+f(r)}{2}-\frac{\Gamma (\alpha +1)}{2(r-a)^{\alpha }}%
\left[ J_{a^{+}}^{\alpha }f(r)+J_{r^{-}}^{\alpha }f(a)\right] \right\vert \\
&\leq &\frac{r-a}{2}\left[ \frac{2}{\alpha +1}\left( 1-\frac{1}{2^{\alpha }}%
\right) \right] ^{1-\frac{1}{q}} \\
&&\times \left\{ \left[ \beta \left( \frac{1}{2},s+1,\alpha +1\right) -\beta
\left( \frac{1}{2},\alpha +1,s+1\right) -\frac{2^{\alpha +s}-1}{\left(
\alpha +s+1\right) 2^{\alpha +s}}\right] \left( \left\vert f^{\prime
}(a)\right\vert ^{q}+\left\vert f^{\prime }(r)\right\vert ^{q}\right)
\right\} ^{\frac{1}{q}}.
\end{eqnarray*}
\end{proof}

\begin{remark}
If we choose $r=b$ in Theorem \ref{6}, we obtain the inequality in (\ref{3})
of Theorem \ref{havva}.
\end{remark}

\begin{remark}
If we choose $r=b$ and $\alpha =1$ in Theorem \ref{6}, we obtain the
inequality in (\ref{2}) of Theorem \ref{havv}.
\end{remark}

\begin{theorem}
\label{66} Let $f:\left[ a,b\right] \subset \lbrack 0,\infty )\rightarrow 
\mathbb{R}
$ be a differentiable mapping on $\left( a,b\right) $ with $a<r$ $\leq b$
such that $f^{\prime }\in L[a,b].$ If $\left\vert f^{\prime }\right\vert
^{q} $ is $s-$concave on $[a,b]$ and $q>1$ with $\frac{1}{p}+\frac{1}{q}=1$,
then the following inequality for fractional integrals holds%
\begin{eqnarray*}
&&\left\vert \frac{f(a)+f(r)}{2}-\frac{\Gamma (\alpha +1)}{2(r-a)^{\alpha }}%
\left[ J_{a^{+}}^{\alpha }f(r)+J_{r^{-}}^{\alpha }f(a)\right] \right\vert \\
&\leq &\frac{r-a}{2^{\frac{2-s}{q}}}\left( \frac{1}{\alpha p+1}\right) ^{%
\frac{1}{p}}\left\vert f^{\prime }\left( \frac{a+r}{2}\right) \right\vert .
\end{eqnarray*}
\end{theorem}

\begin{proof}
From Lemma 1 and using H\"{o}lder inequality, we have 
\begin{eqnarray*}
&&\left\vert \frac{f(a)+f(r)}{2}-\frac{\Gamma (\alpha +1)}{2(r-a)^{\alpha }}%
\left[ J_{a^{+}}^{\alpha }f(r)+J_{r^{-}}^{\alpha }f(a)\right] \right\vert \\
&\leq &\frac{r-a}{2}\int_{0}^{1}\left\vert (1-t)^{\alpha }-t^{\alpha
}\right\vert \left\vert f^{\prime }(r+(a-r)t)\right\vert dt \\
&\leq &\frac{r-a}{2}\left( \underset{0}{\overset{1}{\int }}\left\vert \left(
1-t\right) ^{\alpha }-t^{\alpha }\right\vert ^{p}dt\right) ^{\frac{1}{p}%
}\left( \underset{0}{\overset{1}{\int }}\left\vert f^{\prime
}(r+(a-r)t)\right\vert ^{q}dt\right) ^{\frac{1}{q}}.
\end{eqnarray*}%
Since $\left\vert f^{\prime }\right\vert ^{q}$ is $s-$concave on $[a,b],$we
get%
\begin{equation*}
\underset{0}{\overset{1}{\int }}\left\vert f^{\prime }(r+(a-r)t)\right\vert
^{q}dt\leq 2^{s-1}\left\vert f^{\prime }\left( \frac{a+r}{2}\right)
\right\vert ^{q},
\end{equation*}%
so%
\begin{eqnarray*}
&&\left\vert \frac{f(a)+f(r)}{2}-\frac{\Gamma (\alpha +1)}{2(r-a)^{\alpha }}%
\left[ J_{a^{+}}^{\alpha }f(r)+J_{r^{-}}^{\alpha }f(a)\right] \right\vert \\
&\leq &\frac{r-a}{2^{\frac{2-s}{q}}}\left( \frac{1}{\alpha p+1}\right) ^{%
\frac{1}{p}}\left\vert f^{\prime }\left( \frac{a+r}{2}\right) \right\vert .
\end{eqnarray*}

which completes the proof.
\end{proof}

\begin{corollary}
If we choose $r=b$ in Theorem \ref{66}, we obtain%
\begin{eqnarray}
&&\left\vert \frac{f(a)+f(b)}{2}-\frac{\Gamma (\alpha +1)}{2(b-a)^{\alpha }}%
\left[ J_{a^{+}}^{\alpha }f(b)+J_{b^{-}}^{\alpha }f(a)\right] \right\vert
\label{7} \\
&\leq &\frac{b-a}{2^{\frac{2-s}{q}}}\left( \frac{1}{\alpha p+1}\right) ^{%
\frac{1}{p}}\left\vert f^{\prime }\left( \frac{a+b}{2}\right) \right\vert . 
\notag
\end{eqnarray}
\end{corollary}

\end{document}